\documentclass[final,leqno]{siamltex}

\usepackage{graphicx}
\usepackage{amsmath,amsfonts}
\usepackage[framemethod=tikz]{mdframed}
\usepackage{showkeys}

\def\beq{\begin{equation}}
\def\eeq{\end{equation}}

\newtheorem{assumption}{Assumption}
\newtheorem{example}{Example}
\def\ba{\begin{array}}
\def\ea{\end{array}}
\def\beann{\begin{eqnarray*}}
\def\eeann{\end{eqnarray*}}
\def\bea{\begin{eqnarray}}
\def\eea{\end{eqnarray}}

\def\BT{\begin{theorem}}
\def\ET{\end{theorem}}
\def\BL{\begin{lemma}}
\def\EL{\end{lemma}}
\def\BC{\begin{corollary}}
\def\EC{\end{corollary}}
\def\BE{\begin{example}}
\def\EE{\end{example}}
\def\BD{\begin{definition}}
\def\ED{\end{definition}}
\def\BR{\begin{remark}}
\def\ER{\end{remark}}
\def\BAS{\begin{assumption}}
\def\EAS{\end{assumption}}
\def\BI{\begin{itemize}}
\def\EI{\end{itemize}}

\def\BMP{\begin{minipage}{9.5cm}}
\def\EMP{\end{minipage}}
\def\MPT{\begin{minipage}{11.5cm}}
\def\EPT{\end{minipage}}

\def\theequation{\thesection.\arabic{equation}}

\newtheorem{remark}[theorem]{Remark}

\title{A Generalized Worst-Case Complexity Analysis for Non-Monotone Line Searches}

\author{Geovani N. Grapiglia\thanks{Departamento de Matem\'atica, Universidade Federal do Paran\'a, Centro Polit\'ecnico, Cx. postal 19.081, 81531-980, Curitiba, Paran\'a, Brazil (grapiglia@ufpr.br). This author was supported by the National Council for Scientific and Technological Development - Brazil (grants 401288/2014-5 and 406269/2016-5).}
        \and Ekkehard W. Sachs\thanks{Department of Mathematics, University of Trier, 54286, Trier, Germany (sachs@uni-trier.de).
}}

\date{November 8, 2019}

\begin{document}

\maketitle

\begin{abstract}
We study the worst-case complexity of a non-monotone line search framework that covers a wide variety of known techniques published in the literature. In this framework, the non-monotonicity is controlled by a sequence of nonnegative parameters. We obtain complexity bounds to achieve approximate first-order optimality even when this sequence is not summable. 

\end{abstract}



\pagestyle{myheadings}
\thispagestyle{plain}
\markboth{A Generalized Worst-Case Complexity Analysis for Non-Monotone Line Searches}{G.N. Grapiglia and E.W. Sachs}

\section{Introduction}
\label{intro}

The worst-case complexity analysis of algorithms for non-convex optimization has become a very active research area \cite{NEMI,NES1,NES3}. This type of analysis aims at an estimate for the maximum number of iterations that an algorithm needs to generate an $\epsilon$-approximate critical point of the objective function. The numerical schemes for smooth unconstrained optimization considered so far include line search algorithms \cite{BER,CAR2,GS,NES1,ROY}, trust-region algorithms \cite{CUR,GRA,GRA3,GRAT} and regularization algorithms \cite{BEL,BIR,BIR3,CAR0,CAR1,DUS,GRAN,MAR,NES2,XU}. 

In most of these studies, the algorithms that were analyzed are monotone, that is, they do not allow an increase in the values of the objective function in successive iterations. In this paper we consider a whole family of non-monotone step-size rules and analyze their complexity. This is carried out by using a general algorithmic framework, extending the work in \cite{SAC}.

The framework is built upon a generalized Armijo rule in which the non-monotonicity is controlled by a sequence $\left\{\nu_{k}\right\}$ of non-negative real numbers. 
It was shown in \cite{GS} that, if the sequence $\left\{\nu_{k}\right\}$ is summable, the algorithms in the class take at most $\mathcal{O}(\epsilon^{-2})$ iterations to find $\epsilon$-approximate critical points. 
Here, we relax the summability assumption and provide complexity estimates for the resulting non-monotone schemes. 
As a by-product, we obtain a unified liminf-type global convergence result for non-monotone schemes in which $\nu_{k}\to 0$, covering the non-monotone rules in \cite{GRI} and \cite{ZHA}. Compared to these approaches, the analysis presented here is remarkably simple and our generalized results allow more freedom for the development of new non-monotone line search algorithms. As an example, we design a non-monotone stepsize rule related to the Metropolis rule. 

The paper is organized as follows. In Section 2, we present worst-case complexity estimates. We use these estimates to derive in a new way global convergence results as outlined in Section 3. In Section 4, a Metropolis-based non-monotone rule is motivated and defined. We report preliminary numerical experiments in Section 5. 


\section{Worst-Case Complexity Analysis}
\label{complexity1}

Given a Hilbert space $\left(X,\left\langle\,.\,,\,.\,\right\rangle\right)$, we consider the minimization problem
\begin{equation}
\min_{x\in X}\,f(x),
\label{eq:extra1.1}
\end{equation}
where $f:X\to\mathbb{R}$ is Fr\'echet differentiable. We shall denote the gradient of $f$ at $x\in X$ by $\nabla f(x)$. Furthermore, given $x_{k}\in X$, we call $d_{k}\in X$ a descent direction for $f$ at $x_{k}$ if $\left\langle \nabla f(x_{k}),d_{k}\right\rangle<0$. Finally, we shall denote the norm induced by the inner product $\left\langle\,.\,,\,.\,\right\rangle$ by $\|\cdot\|$.

We will consider the following general descent algorithm with a non-monotone Armijo line search, which is a slight modification of the scheme proposed by Sachs and Sachs \cite{SAC}.

\noindent\textbf{Algorithm 1. (General Non-monotone Descent Algorithm)}
\\[0.2cm]
\noindent\textbf{Step 0} Given $x_{0}\in X$, $\alpha_{0}>0$ and $\beta,\rho\in (0,1)$, set $k:=0$.\\
\noindent\textbf{Step 1} Compute a descent direction $d_{k}\in X$ for $x_{k}$.\\
\noindent\textbf{Step 2.1} Set $l:=0$.\\
\noindent\textbf{Step 2.2} Choose $\nu_{k,l}\geq 0$. If 
                         \begin{equation}
                         f(x_{k}+\alpha_{k}\beta^{l}d_{k})\leq f(x_{k})+\rho\alpha_{k}\beta^{l}\left\langle \nabla f(x_{k}),d_{k}\right\rangle +\nu_{k,l}
                         \label{eq:2.1}
                         \end{equation}
                         set $l_{k}=l$, $\nu_{k}=\nu_{k,l_{k}}$ and go to Step 3. Otherwise, set $l:=l+1$ and repeat Step 2.2.\\
\noindent\textbf{Step 3} Set $x_{k+1}=x_{k}+\alpha_{k}\beta^{l_{k}}d_{k}$, $\alpha_{k+1}=\alpha_{k}\beta^{l_{k}-1}$, $k:=k+1$ and go to Step 1. 

\begin{remark}
The root of the non-monotone term $f(x_{k})+\nu_{k,l}$ in Algorithm 1 can be traced back to \cite{LI} and \cite{CRUZ}. In addition, a trust-region with line search using a similar non-monotone term has been proposed in \cite{TARZAN}.
\end{remark}

\begin{remark}
The difference between Algorithm 1 and the general scheme in \cite{SAC} is that at any given iteration $k$, instead of using a fixed non-monotone term $\nu_{k}$, we allow it to change within the line search procedure. This flexibility allows to cover the non-monotone rule described in Section 4.
\end{remark}

To analyze the worst-case complexity of Algorithm 1, we shall consider the following assumptions:

\begin{itemize}
\item[\textbf{A1}] The objective function $f:X\to\mathbb{R}$ is Fr\'echet differentiable and its gradient $\nabla f:X\to X$ is Lipschitz continuous with Lipschitz constant $L>0$.

\item[\textbf{A2}] There exists $f_{low}\in\mathbb{R}$ such that $f(x)\geq f_{low}$ for all $x\in X$.

\item[\textbf{A3}] For all $k$,
                   \begin{equation*}
                   \left\langle \nabla f(x_{k}),d_{k}\right\rangle\leq -c_{1}\|\nabla f(x_{k})\|^{2}\quad\text{and}\quad \|d_{k}\|\leq c_{2}\|\nabla f(x_{k})\|
                   \end{equation*}
                   for some constants $c_{1},c_{2}>0$. 
\end{itemize}

\begin{lemma}
\label{lem:2.1}
Suppose $f:X\to\mathbb{R}$ is Fr\'echet differentiable and its gradient $\nabla f:X\to X$ is Lipschitz continuous with Lipschitz constant $L>0$:
\begin{equation*}
\|\nabla f(x)-\nabla f(y)\|\leq L\|x-y\|,\quad\forall\,x,y\in X.
\end{equation*} 
Then,
\begin{equation}
f(y)\leq f(x)-\left\langle \nabla f(x),y-x\right\rangle+\dfrac{L}{2}\|y-x\|^{2}\quad\forall x,y\in X.
\label{eq:extra1}
\end{equation}
\end{lemma}

\begin{proof}
See, for example, Theorem 1.2.22 in \cite{SUN}.
\end{proof}

The next lemma provides a lower bound on $\alpha_{k}$.
\begin{lemma}
\label{lem:2.2}
Suppose that A1 and A3 hold. Then, for all $k$,
\begin{equation}
\alpha_{k}\geq\min\left\{\alpha_{0},\dfrac{2(1-\rho)c_{1}}{Lc_{2}^{2}}\right\}\equiv\bar{\alpha}.
\label{eq:2.2}
\end{equation}
\end{lemma}

\begin{proof}
Since $\nu_{k,l}\geq 0$ for all $k$ and $l$, the result can be shown as in the proof of Lemma 2 in \cite{GS}.
\end{proof}

The first theorem gives an upper bound on the total number of function evaluations after $k\geq 1$ iterations.
\begin{theorem}
\label{lem:stela1}
Suppose that A1 and A3 hold and let $N_{k}$ be the total number of function evaluations up to the $k$-th iteration of Algorithm 1. Then,
\begin{equation}
N_{k}\leq 2(k+1)+\dfrac{1}{\log(\beta)}\left[\log(\bar{\alpha})-\log(\alpha_{0})\right], 
\label{eq:stela2.1}
\end{equation}
where $\bar{\alpha}$ is defined in Lemma \ref{lem:2.2}.
\end{theorem}
\begin{proof}
Theorem 3 in \cite{GS} applies here, since the proof only uses $\alpha_{k+1}=\beta^{\ell_{k}-1}\alpha_{k}$ and the bound $\alpha_{k}\geq\bar{\alpha}$ for all $k$.
\end{proof}

\begin{remark}
\label{rem:stela1}
From (\ref{eq:stela2.1}) we see that in Algorithm 1 the average number of function evaluations per iteration, up to the $k$-th iteration, is asymptotically bounded by $2$:
\begin{equation*}
\frac{N_{k}}{k} \leq 2 \left(1 + \frac{1}{k}\right) +\frac{1}{k} \dfrac{\log(\bar{\alpha})-\log(\alpha_{0})}{\log(\beta)}.
\end{equation*}
\end{remark}

Now, define
\begin{equation}
\kappa_{c}=\min\left\{\rho\beta\alpha_{0}c_{1},\dfrac{2\beta\rho(1-\rho)c_{1}^{2}}{Lc_{2}^{2}}\right\}.
\label{eq:2.3}
\end{equation}
With respect to sequence $\left\{\nu_{k}\right\}_{k=0}^{+\infty}$ that controls the amount of the non-monotonicity, we shall consider the following assumption:

\begin{itemize}
\item[\textbf{A4}] $\lim_{T\to +\infty}\frac{1}{T}\sum_{k=0}^{T-1}\nu_{k}=0$.
\end{itemize}

Note that, if $\sum_{k=0}^{+\infty}\nu_{k}<+\infty$, then A4 is satisfied. However, A4 also may be satisfied for sequences that are not summable. An example is $\nu_{k}=M/(k+1)$, with $M>0$, for which $\sum_{k=0}^{+\infty}\nu_{k}=+\infty$ but 
\begin{equation*}
\lim_{T\to +\infty}\dfrac{1}{T}\sum_{k=0}^{T-1}\nu_{k}\leq\lim_{T\to +\infty}\dfrac{M}{T}\left[\ln(T)+1\right]=0.
\end{equation*}
Therefore, the complexity analysis presented here includes non-monotone terms that were not covered by the analysis in \cite{GS}.

Given $\epsilon>0$, under the assumption A4, we shall denote by $T_{0}(\epsilon)$ any non-negative integer such that
\begin{equation}
T\geq T_{0}(\epsilon) \Longrightarrow \frac{1}{T}\sum_{k=0}^{T-1}\,\nu_{k} \leq\dfrac{\kappa_{c}\epsilon^{2}}{2},
\label{eq:extra2.2}
\end{equation}
where $\kappa_{c}$ is given by (\ref{eq:2.3}).

Our next theorem establishes an upper bound on the number of iterations necessary for Algorithm 1 generate $x_{k}$ such that $\|\nabla f(x_{k})\|\leq\epsilon$. Using (\ref{eq:extra2.2}), the proof follows by adapting the proof of Theorem 1 in \cite{GS}.
\begin{theorem}
\label{thm:2.3}
Suppose that A1-A4 hold and let the sequence $\left\{x_{k}\right\}_{k=0}^{+\infty}$ be ge\-ne\-ra\-ted by Algorithm 1. If 
\begin{equation}
T\geq\max\left\{T_{0}(\epsilon), \dfrac{2(f(x_{0})-f_{low})}{\kappa_{c}\epsilon^{2}}\right\},
\label{eq:2.4}
\end{equation}
then
\begin{equation}
\min_{k=0,\ldots,T-1}\|\nabla f(x_{k})\|\leq\epsilon.
\label{eq:2.5}
\end{equation}
\end{theorem}

\begin{proof}
It follows from (\ref{eq:2.1}), A3 and Lemma \ref{lem:2.2} that
\begin{eqnarray}
\nu_{k}+f(x_{k})-f(x_{k+1}) &\geq & \rho\alpha_{k}\beta^{l_{k}}\left(-\left\langle \nabla f(x_{k}),d_{k} \right\rangle\right)\nonumber\\
                           & \geq & \rho\beta\alpha_{k+1}c_{1}\|\nabla f(x_{k})\|^{2}\nonumber\\
                         &\geq & \kappa_{c}\|\nabla f(x_{k})\|^{2},
\label{eq:2.6}
\end{eqnarray}
where $\kappa_{c}$ is defined in (\ref{eq:2.3}). Summing up these inequalities for $k=0,\ldots,T-1$, and using A2, we get
\begin{eqnarray*}
\sum_{k=0}^{T-1}\kappa_{c}\|\nabla f(x_{k})\|^{2}\leq 
f(x_{0})-f(x_{T})+\sum_{k=0}^{T-1}\,\nu_{k}
\leq  f(x_{0})-f_{low}+\sum_{k=0}^{T-1}\,\nu_{k}.
\end{eqnarray*}
Consequently, 
\begin{equation*}
\kappa_{c}T\min_{k=0,\ldots,T-1}\|\nabla f(x_{k})\|^{2}\leq f(x_{0})-f_{low}+\sum_{k=0}^{T-1}\,\nu_{k},
\end{equation*}
which gives
\begin{equation}
\min_{k=0,\ldots,T-1}\|\nabla f(x_{k})\|^{2}\leq \dfrac{f(x_{0})-f_{low}}{\kappa_{c}T}+\dfrac{1}{\kappa_{c}T}\sum_{k=0}^{T-1}\,\nu_{k}.
\label{eq:2.7}
\end{equation}
Since (\ref{eq:2.4}) holds, we have $T\geq T_{0}(\epsilon)$, and so it follows from (\ref{eq:extra2.2}) that
\begin{equation}
\dfrac{1}{\kappa_{c}T}\sum_{k=0}^{T-1}\,\nu_{k}\leq\dfrac{\epsilon^{2}}{2}.
\label{eq:2.8}
\end{equation}
On the other hand, also by (\ref{eq:2.4}) we have 
\begin{equation}
\dfrac{f(x_{0})-f_{low}}{\kappa_{c}T}\leq\dfrac{\epsilon^{2}}{2}.
\label{eq:2.9}
\end{equation}
Combining (\ref{eq:2.7}), (\ref{eq:2.8}) and (\ref{eq:2.9}), we have
\begin{equation*}
\min_{k=0,\ldots,T-1}\|\nabla f(x_{k})\|^{2}\leq\dfrac{\epsilon^{2}}{2}+\dfrac{\epsilon^{2}}{2}=\epsilon^{2},
\end{equation*}
which gives (\ref{eq:2.5}).
\end{proof}

An important class of non-monotone schemes is the one that corresponds to $\left\{\nu_{k}\right\}$ summable. As mentioned in the Introduction, it includes, for example, the non-monotone rule of Zhang and Hager \cite{ZHA} and the non-monotone rule of Ahookhosh, Amini and Bahrami \cite{AHO} (for details, see Section 6 in \cite{SAC}). For this class, Theorem \ref{thm:2.3} has the following consequence.
\begin{corollary}
\label{cor:2.1}
Suppose that A1-A3 hold and that $\sum_{k=0}^{+\infty}\,\nu_{k}<+\infty$. Let $\left\{x_{k}\right\}_{k=0}^{+\infty}$ be a sequence generated by Algorithm 1. Given $\epsilon\in (0,1)$, if 
\begin{equation} \label{eq:2.3a}
T\geq 2\max\left\{\sum_{k=0}^{+\infty}\,\nu_{k},\,f(x_{0})-f_{low}\right\}\kappa_{c}^{-1}\epsilon^{-2},
\end{equation}
then
\begin{equation}
\min_{k=0,\ldots,T-1}\|\nabla f(x_{k})\|\leq\epsilon.
\label{eq:extra2.3}
\end{equation}
\end{corollary}
\begin{proof}
Note that
\begin{equation*}
0\leq\frac{1}{T}\sum_{k=0}^{T-1}\nu_{k}\leq\frac{1}{T}\sum_{k=0}^{+\infty}\nu_{k},\quad \text{for all}\,\,T\geq 1.
\end{equation*}
Since $\left\{\nu_{k}\right\}$ is summable, it follows that A4 is sa\-tis\-fied. Moreover, 
\begin{equation*}
T\geq\dfrac{2\left(\sum_{k=0}^{+\infty}\nu_{k}\right)}{\kappa_{c}\epsilon^{2}}\Longrightarrow \dfrac{\kappa_{c}\epsilon^{2}}{2}\geq \frac{1}{T}\left(\sum_{k=0}^{+\infty}\nu_{k}\right)\geq \frac{1}{T}\left(\sum_{k=0}^{T-1}\nu_{k}\right).
\end{equation*}
Therefore, (\ref{eq:extra2.2}) holds for
\begin{equation*}
T_{0}(\epsilon)=\dfrac{2\sum_{k=0}^{+\infty}\nu_{k}}{\kappa_{c}\epsilon^{2}},
\end{equation*}
and (\ref{eq:2.3a}) can be rewritten as
\begin{equation*}
T\geq\max\left\{T_{0}(\epsilon),\dfrac{2(f(x_{0})-f_{low})}{\kappa_{c}\epsilon^{2}}\right\}.
\end{equation*}
Thus, by Theorem \ref{thm:2.3}, (\ref{eq:extra2.3}) must be true.
\end{proof}

When $\sum_{k=0}^{+\infty}\nu_{k}<+\infty$, Corollary \ref{cor:2.1} gives a worst-case complexity bound of $\mathcal{O}(\epsilon^{-2})$ iterations, which agrees with the bound established in \cite{GS}. The next result allows us to obtain worst-case complexity estimates even when $\left\{\nu_{k}\right\}$ is not summable.

\begin{corollary}
\label{cor:2.2}
Suppose that A1-A3 hold and that $\nu_{k}\to 0$. Let constant $C>0$ such that $\nu_{k}\leq C$ for all $k$ and, given $\delta>0$, let $k_{0}(\delta)$ be a positive integer such that $\nu_{k}\leq \delta$ if $k\geq k_{0}(\delta)$. Then, for any sequence $\left\{x_{k}\right\}_{k=0}^{+\infty}$ generated by Algorithm 1, if
\begin{equation}
T\geq \max\left\{\dfrac{2k_{0}(\delta/2)C}{\delta},1+k_{0}(\delta/2),\dfrac{2(f(x_{0})-f_{low})}{\kappa_{c}\epsilon^{2}}\right\}
\label{eq:2.10}
\end{equation}
for $\delta=\kappa_{c}\epsilon^{2}/2$, it follows that
\begin{equation}
\min_{k=0,\ldots,T-1}\|\nabla f(x_{k})\|\leq\epsilon.
\label{eq:2.11}
\end{equation}
In particular, if $\nu_{k}=M/k$ for all $k$, with $M>0$ constant, then (\ref{eq:2.11}) holds if 
\begin{equation}
T\geq \max\left\{\dfrac{16M^{2}}{\kappa_{c}^{2}\epsilon^{4}},1+\dfrac{4M}{\kappa_{c}\epsilon^{2}},\dfrac{2(f(x_{0})-f_{low})}{\kappa_{c}\epsilon^{2}}\right\}.
\label{eq:2.11b}
\end{equation}
\end{corollary}

\begin{proof}
Given $\delta>0$, if
\begin{equation*}
T\geq\max\left\{\dfrac{2 k_{0}(\delta/2)C}{\delta},1+k_{0}(\delta/2)\right\}
\end{equation*}
we have
\begin{eqnarray*}
\dfrac{1}{T}\sum_{k=0}^{T-1}\,\nu_{k}& = &\dfrac{1}{T}  \left(\sum_{k=0}^{k_{0}(\delta/2)-1}\,\nu_{k}\right)+ \dfrac{1}{T} \left(\sum_{k=k_{0}(\delta/2)}^{T-1}\,\nu_{k}\right)\\
 &\leq &  \dfrac{1}{T} \left(\sum_{k=0}^{k_{0}(\delta/2)-1}\,C\right)+\dfrac{1}{T} \left(\sum_{k=k_{0}(\delta/2)}^{T-1}\,\dfrac{\delta}{2}\right)\\
 &\leq & \dfrac{1}{T} \left(\sum_{k=0}^{k_{0}(\delta/2)-1}\,C\right)+\dfrac{1}{T} \left(\sum_{k=0}^{T-1}\,\dfrac{\delta}{2}\right)\\
& \leq & \dfrac{k_{0}(\delta/2)C}{T}+\dfrac{\delta}{2}\\
 &\leq &\delta.
\end{eqnarray*}
Therefore, the assumption A4 is satisfied and (\ref{eq:extra2.2}) holds for  
\begin{equation*}
T_{0}(\epsilon)=\max\left\{\dfrac{2 k_{0}(\delta/2)C}{\delta},1+k_{0}(\delta/2)\right\},
\end{equation*}
with $\delta=\kappa_{c}\epsilon^{2}/2$.
Consequently, if (\ref{eq:2.10}) holds, then (\ref{eq:2.4}) is true and the conclusion comes directly from Theorem \ref{thm:2.3}.
Finally, suppose that $\nu_{k}=M/k$ for all $k$. Then, $\nu_{k}\to +\infty$, $\nu_{k}\leq M$ for all $k$ and, given $\delta>0$, 
\begin{equation*}
\nu_{k}=\dfrac{M}{k}\leq\delta\Longleftrightarrow k\geq\dfrac{M}{\delta}.
\end{equation*}
Hence, in this case, we have
\begin{equation*}
k_{0}(\delta)=\dfrac{M}{\delta}\quad\text{and}\quad C=M.
\end{equation*}
Therefore, the condition (\ref{eq:2.10}) becomes (\ref{eq:2.11b}).
\end{proof}
\begin{remark}
Consider $\nu_{k}=\epsilon/k$ for all $k\geq 1$, with $\epsilon\in (0,1)$. In this case, even though $\sum_{k=0}^{+\infty} \nu_k = +\infty$, it follows from Corollary \ref{cor:2.2} (with $M=\epsilon$) that Algorithm 1 takes at most $\mathcal{O}(\epsilon^{-2})$ iterations to generate $x_{k}$ such that $\|\nabla f(x_{k})\|\leq\epsilon$.
\end{remark}

A worst-case complexity bound of $\mathcal{O}\left(\epsilon^{-2}\right)$ also can be obtained for variants of Algorithm 1 characterized by the following assumption:
\begin{itemize}
\item[\textbf{A4'}] $\nu_{k}=\text{o}\left(\|\nabla f(x_{k})\|^{2}\right)$ as $k\to +\infty$, i.e., for all $\delta>0$, there exists $n_{0}\in\mathbb{N}$ such that
\begin{equation*}
\nu_{k}\leq\delta \|\nabla f(x_{k})\|^{2},\quad\forall k\geq n_{0}.
\end{equation*}
\end{itemize}
Under assumption A4', for any $\delta>0$, we can define the number
\begin{equation}
n_{0}(\delta)=\min\left\{n_{0}\in\mathbb{N}\,|\,\nu_{k}\leq\delta\|\nabla f(x_{k})\|^{2},\,\,\forall k\geq n_{0}\right\}.
\label{eq:alencar1}
\end{equation} 
One example of sequence $\left\{\nu_{k}\right\}$ satisfying A4' is
\begin{equation*}
\nu_{0}=0\quad\text{and}\quad \nu_{k}=k^{-1}\|\nabla f(x_{k})\|^{2}\quad\forall k\geq 1.
\end{equation*}

The next lemma gives a finite upper bound of $\mathcal{O}(\epsilon^{-2})$ for the \textit{total number of iterations} of Algorithm 1 in which $\|\nabla f(x_{k})\|>\epsilon$ for a given $\epsilon>0$.
\begin{lemma}
\label{lem:extra4}
Suppose that A1-A3 hold and let sequence $\left\{x_{k}\right\}_{k=0}^{+\infty}$ be generated by Algorithm 1. Given $\epsilon>0$, if A4' holds, then the number of elements of the set
\begin{equation}
\Omega_{\epsilon}=\left\{k\,|\,\|\nabla f(x_{k})\|>\epsilon\right\}
\label{eq:extra24}
\end{equation}
is bounded as follows
\begin{equation}
|\Omega_{\epsilon}|\leq k_{1}+\left[\dfrac{2\left(f(x_{0})-f_{low}+\sum_{i=0}^{k_{1}-1}\nu_{i}\right)}{\kappa_{c}}\right]\epsilon^{-2},
\label{eq:extra25}
\end{equation}
where $k_{1}=n_{0}(\frac{\kappa_{c}}{2})$ is independent of $\epsilon$, with $\kappa_{c}$ and $n_{0}(\cdot)$ defined in (\ref{eq:2.3}) and (\ref{eq:alencar1}), respectively.
\end{lemma}

\begin{proof}
By A4' and (\ref{eq:alencar1}), $k_{1}$ is well-defined and 
\begin{equation*}
\nu_{k}\leq\dfrac{\kappa_{c}}{2}\|\nabla f(x_{k})\|^{2},\quad\forall k\geq k_{1}.
\end{equation*}

Thus, it follows from (\ref{eq:2.6}) that

\begin{eqnarray*}
\kappa_{c}\|\nabla f(x_{k})\|^{2}&\leq & f(x_{k})-f(x_{k+1})+\nu_{k}\\
                                 &\leq & f(x_{k})-f(x_{k+1})+\dfrac{\kappa_{c}}{2}\|\nabla f(x_{k})\|^{2},\,\,\forall k\geq k_{1},
\end{eqnarray*}

which implies that
\begin{equation}
\dfrac{\kappa_{c}}{2}\|\nabla f(x_{k})\|^{2}\leq f(x_{k})-f(x_{k+1}),\quad\forall k\geq k_{1}.
\label{eq:extra26}
\end{equation}
Given $0\leq s<t$, let us define
\begin{equation}
\Omega_{\epsilon}(s,t)=\left\{s\leq k\leq t\,|\,\nabla f(x_{k})\|>\epsilon\right\}.
\label{eq:extra27}
\end{equation}
For all $t>k_{1}$, it follows from (\ref{eq:extra26}) that
\begin{equation*}
\dfrac{\kappa_{c}}{2}\epsilon^{2}\leq f(x_{k})-f(x_{k+1}),\quad\forall k\in\Omega_{\epsilon}(k_{1},t).
\end{equation*}
Therefore,
\begin{eqnarray*}
|\Omega_{\epsilon}(k_{1},t)|\dfrac{\kappa_{c}\epsilon^{2}}{2}&=&\sum_{k\in\Omega_{\epsilon}(k_{1},t)}\dfrac{\kappa_{c}\epsilon^{2}}{2}\\
&\leq &\sum_{k\in\Omega_{\epsilon}(k_{1},t)}f(x_{k})-f(x_{k+1})\\
&\leq &\sum_{k=k_{1}}^{t}f(x_{k})-f(x_{k+1})\\
& = & f(x_{k_{1}})-f(x_{t+1})\\
&\leq & f(x_{k_{1}})-f_{low},
\end{eqnarray*}
and so
\begin{equation}
|\Omega_{\epsilon}(k_{1},t)|\leq\left[\dfrac{2\left(f(x_{k_{1}})-f_{low}\right)}{\kappa_{c}}\right]\epsilon^{-2}.
\label{eq:extra28}
\end{equation}
Since $t>k_{1}$ is arbitrary, by (\ref{eq:extra27}), (\ref{eq:extra28}) and (\ref{eq:extra24}), we get
\begin{equation}
|\Omega_{\epsilon}|\leq k_{1}+|\Omega_{\epsilon}(k_{1},+\infty)|\leq k_{1}+\left[\dfrac{2(f(x_{k_{1}})-f_{low})}{\kappa_{c}}\right]\epsilon^{-2}.
\label{eq:extra29}
\end{equation}
Finally, notice that

\begin{equation}
f(x_{k_{1}})\leq f(x_{0})+\sum_{i=0}^{k_{1}-1}\nu_{i}.
\label{eq:extra30}
\end{equation}

Thus, (\ref{eq:extra25}) follows directly from (\ref{eq:extra29}) and (\ref{eq:extra30}).
\end{proof}

\section{Global Convergence Results}
The next theorem comes as a by-product from the previous complexity estimates and yields a convergence result which simplifies known proofs substantially and generalizes other non-monotone step-size rules.
\begin{theorem}
\label{thm:3.1}
Suppose that A1-A3 hold and let the sequence $\left\{x_{k}\right\}_{k=0}^{+\infty}$ be generated by Algorithm 1. If $\nu_{k}\to 0$ as $k\to +\infty$, then either there exists $\bar{k}$ such that $\nabla f(x_{\bar{k}})=0$ or
\begin{equation}
\liminf_{k\to +\infty}\,\|\nabla f(x_{k})\|=0.
\label{eq:2.12}
\end{equation}
\end{theorem}
\begin{proof}
Let $\epsilon >0$. Since $\nu_{k}\to 0$ as $k\to +\infty$, there exist constants $C$ and $k_{0}(\frac{\kappa_{c}\epsilon^{2}}{4})>0$ such that $\nu_{k}\leq C$ for all $k$, and $\nu_{k}\leq\kappa_{c}\epsilon^{2}/4$ for all $k\geq k_{0}(\frac{\kappa_{c}\epsilon^{2}}{4})$. Thus, from Corollary \ref{cor:2.2}, if 
\begin{equation}
T\geq \max\left\{\dfrac{4k_{0}(\frac{\kappa_{c}\epsilon^{2}}{4})C}{\kappa_{c}\epsilon^{2}},1+k_{0}\left(\frac{\kappa_{c}\epsilon^{2}}{4}\right),\dfrac{2(f(x_{0})-f_{low})}{\kappa_{c}\epsilon^{2}}\right\}
\label{eq:2.14}
\end{equation}
then
\begin{equation*}
\min_{k=0,\ldots,T-1}\|\nabla f(x_{k})\|\leq\epsilon.
\end{equation*}
As $\epsilon>0$ is arbitrary, this proves that
\begin{equation*}
\lim_{T\to +\infty}\,\left(\min_{k=0,\ldots,T-1}\|\nabla f(x_{k})\|\right)=0.
\end{equation*}
Therefore, either there exists $\bar{k}$ for which $\|\nabla f(x_{\bar{k}})\|=0$ or (\ref{eq:2.12}) is true.
\end{proof}

More importantly, our analysis provides a unified global convergence proof for many non-monotone methods based on the method proposed in \cite{GRI}, which is one of the most used non-monotone line search algorithms. It corresponds to the modified Armijo rule
\begin{equation}
f(x_{k}+\alpha_{k}\beta^{l_{k}}d_{k})\leq \max_{0\leq j\leq m(k)}\,f(x_{k-j})+\rho\alpha_{k}\beta^{l_{k}}\langle \nabla f(x_{k}),d_{k}\rangle,
\label{eq:alencar2}
\end{equation}
for a suitable choice of $m(k)$.
Notice that this rule can be written in the form (\ref{eq:2.1}) with 
\begin{equation*}
\nu_{k,l}\equiv\nu_{k}=\max_{0\leq j\leq m(k)}\,f(x_{k-j})-f(x_{k}).
\end{equation*} 

\begin{corollary}
\label{cor:2.1a}
Suppose that A1-A3 hold and let sequence $\left\{x_{k}\right\}_{k=0}^{+\infty}$ be generated by Algorithm 1 where (\ref{eq:2.1}) is replaced by
\begin{equation}
f(x_{k}+\alpha_{k}\beta^{l_{k}}d_{k})\leq R_{k}+\rho\alpha_{k}\beta^{l_{k}}\langle\nabla f(x_{k}),d_{k}\rangle
\label{eq:2.16}
\end{equation}
with 
\begin{equation}
f(x_{k})\leq R_{k}\leq \max_{0\leq j\leq m(k)}\,f(x_{k-j})
\label{eq:2.17}
\end{equation}
where $m(0)=0$ and $0\leq m(k)\leq \min\left\{m(k-1)+1,N\right\}$, for a user-defined $N\in\mathbb{N}$. 
If
\begin{equation*}
\left\{x\in\mathbb{R}^{n}\,|\,f(x)\leq f(x_{0})\right\}
\end{equation*}
is bounded, then either exists $\bar{k}$ such that $\nabla f(x_{\bar{k}})=0$ or
\begin{equation*}
\liminf_{k\to +\infty}\,\|\nabla f(x_{k})\|=0.
\end{equation*}
\end{corollary}
\begin{proof}
As in the proof of Lemma 2 in \cite{AMI}, by (\ref{eq:2.17}), one can show that
\begin{equation}
\lim_{k\to +\infty}\,f(x_{k})=\lim_{k\to +\infty}\max_{0\leq j\leq m(k)}\,f(x_{k-j}).
\label{eq:2.15}
\end{equation}
Hence for $\nu_k = R_k - f(x_k)$ we obtain
\begin{equation*}
\lim_{k\to +\infty}\,\nu_{k} = \lim_{k\to +\infty}\,  R_k - f(x_k)   \leq \lim_{k\to +\infty}\,  \max_{0\leq j\leq m(k)}\,f(x_{k-j}) - f(x_k) =0.
\end{equation*}
Therefore, the result follows directly from Theorem \ref{thm:3.1}. 
\end{proof}
This generalized convergence result includes, for example, the non-monotone methods in \cite{AHO,AHOO,AMI,NOS}. The worst-case complexity of these methods, however, depends on how fast $\nu_{k}=R_{k}-f(x_{k})$ converges to zero. Due to the $\max\left\{\cdot\right\}$ on the right-hand side of (\ref{eq:2.10}), the best iteration-complexity bound that one can get from Corollary \ref{cor:2.2} is $\mathcal{O}(\epsilon^{-2})$. This is exactly the complexity obtained by Cartis, Sampaio and Toint \cite{CAR2} for a non-monotone method based on rule (\ref{eq:alencar2}).

Notice that Theorem \ref{thm:3.1} gives a liminf-type convergence result. An improved lim-type result can be obtained for variants of Algorithm 1 characterized by assumption A4'. Indeed, from the complexity estimate given in Lemma \ref{lem:extra4}, we can establish the global convergence of Algorithm 1 with the same argument used to prove Corollary 2.1 in \cite{MAR}.

\begin{theorem}
\label{thm:extra5}
Suppose that A1-A3 hold and let the sequence $\left\{x_{k}\right\}_{k=0}^{+\infty}$ be generated by Algorithm 1. If A4' also holds, then
\begin{equation}
\lim_{k\to +\infty}\|\nabla f(x_{k})\|=0.
\label{eq:extra31}
\end{equation}
\end{theorem}

\begin{proof}
Suppose that (\ref{eq:extra31}) does not hold. Then, there exists $\epsilon>0$ and a subsequence $\left\{x_{k_{j}}\right\}_{j=0}^{+\infty}$ of $\left\{x_{k}\right\}_{k=0}^{+\infty}$ such that
\begin{equation*}
\|\nabla f(x_{k_{j}})\|>\epsilon,\quad\forall j\in\mathbb{N}.
\end{equation*}
This means that the corresponding set $\Omega_{\epsilon}=\left\{k\,|\,\|\nabla f(x_{k})\|>\epsilon\right\}$ is infinite, contradicting Lemma \ref{lem:extra4}.
\end{proof}

\section{A Metropolis-Based Non-Monotone Rule}

One of the core ideas of non-monotone rules is to allow the iterates to escape from local minimizers and to increase the probability of finding global minimizers. In the context of derivative-free heuristics for global optimization, Simulated Annealing \cite{MET,KIR,LOC} is one of the most efficient schemes. At the $k$th iteration of a simulated annealing algorithm, the acceptance or rejection of a candidate point $x_{k}^{+}$ is usually done by the \textit{Metropolis rule}: given a uniform random number $p_k \in [0,1]$, the next iterate is set as
\begin{equation}
x_{k+1}=\left\{\begin{array}{ll} x_{k}^{+},& \text{if}\,\,p_k \leq\min\left\{1,\text{exp}\left(-\dfrac{f(x_{k}^{+})-f(x_{k})}{\tau_{k}}\right)\right\},\\
                                 x_{k},&\text{otherwise},
              \end{array}
        \right.
\label{eq:4.1}
\end{equation}
where $\tau_{k}>0$ for all $k$, with $\tau_{k}\to 0$. By rule (\ref{eq:4.1}), if $f(x_{k}^{+})\leq f(x_{k})$ then $x_{k}^{+}$ is always accepted, i.e., $x_{k+1}=x_{k}^{+}$. However, the candidate point $x_{k}^{+}$ also can be accepted when $f(x_{k}^{+})>f(x_{k})$, allowing the iterates to escape from local minimizers. The larger the difference $f(x_{k}^{+})-f(x_{k})>0$ is, the smaller is the probability to accept $x_{k}^{+}$. Since $\tau_{k}\to 0$, the probability of accepting $x_{k}^{+}$ when $f(x_{k}^{+})>f(x_{k})$ also goes to zero when $k\to +\infty$.

Back to Algorithm 1, notice that the bigger is the non-monotone parameter $\nu_{k,l}$, the bigger is the chance to accept a candidate point $x_{k,l}^{+}=x_{k}+\alpha_{k}\beta^{l}d_{k}$ with $f(x_{k,l}^{+})>f(x_{k})$. Thus, we can try to mimic the Metropolis acceptance rule by choosing $\nu_{k,l}$ as follows:

\noindent\textbf{Step 2.1} Set $l:=0$.\\
\noindent\textbf{Step 2.2} Compute $x_{k,l}^{+}=x_{k}+\alpha_{k}\beta^{l}d_{k}$ and define
\begin{equation}
\nu_{k,l}=\sigma\text{exp}\left(-\dfrac{\max\left\{\theta,f(x_{k,l}^{+})-f(x_{k})\right\}}{\tau_{k}}\right)
\label{eq:4.2}
\end{equation}
for some constants $\sigma,\theta>0$ independent of $k$ and $l$, with $\tau_{k}=1/\ln(k+1)$. If
\begin{equation*}
f(x_{k,l}^{+})\leq f(x_{k})+\rho\alpha_{k}\beta^{l}\langle\nabla f(x_{k}),d_{k}\rangle +\nu_{k,l}
\end{equation*}
set $l_{k}=l$ and $\nu_{k}=\nu_{k,l_{k}}$. Otherwise, set $l:=l+1$ and repeat Step 2.2.

The next two theorems establish complexity bounds of $\mathcal{O}(\epsilon^{-2})$ and $\mathcal{O}(\epsilon^{-\frac{2(1+\theta)}{\theta}})$ for Algorithm 1, when $\theta>1$ and $\theta\in (0,1]$, respectively.
\begin{theorem}
\label{thm:5.1}
Suppose that A1-A3 hold and let the sequence $\left\{x_{k}\right\}_{k=0}^{+\infty}$ be generated by Algorithm 1 with $\nu_{k,l}$ defined by (\ref{eq:4.2}). Given $\epsilon>0$, if $\theta > 1$ and
\begin{equation} \label{eq:4.3a}
T\geq 2\max\left\{\sigma\sum_{k=0}^{+\infty}\,\frac{1}{(k+1)^{\theta}},\,f(x_{0})-f_{low}\right\}\kappa_{c}^{-1}\epsilon^{-2},
\end{equation}
then
\begin{equation}
\min_{k=0,\ldots,T-1}\|\nabla f(x_{k})\|\leq\epsilon.
\label{eq:4.4}
\end{equation}
\end{theorem}

\begin{proof}
By (\ref{eq:4.2}), for all $k$ we have
\begin{equation}
\nu_{k}=\sigma e^{-\max\left\{\theta, f(x_{k+1})-f(x_{k})\right\}\text{ln}(k+1)}\leq\sigma \left(\dfrac{1}{k+1}\right)^{\theta}.
\label{eq:extra32}
\end{equation}
Thus, 
\begin{equation*}
\sum_{k=0}^{+\infty}\nu_{k}=\sigma\sum_{k=0}^{+\infty}\left(\dfrac{1}{k+1}\right)^{\theta}<+\infty,
\end{equation*}
and Corollary 1 yields the result.
\end{proof}

\begin{theorem}
\label{thm:extra6}
Suppose that A1-A3 hold and let sequence $\left\{x_{k}\right\}_{k=0}^{+\infty}$ be generated by Algorithm 1 with $\nu_{k,\ell}$ defined by (\ref{eq:4.2}). Given $\epsilon\in (0,1)$, if $\theta\in (0,1]$ and
\begin{equation}
T\geq\max\left\{\left(\dfrac{4}{\kappa_{c}}\right)^{\frac{1+\theta}{\theta}}\sigma^{\frac{1}{\theta}},1+\left(\dfrac{4\sigma}{\kappa_{c}}\right)^{\frac{1}{\theta}},\dfrac{2(f(x_{0})-f_{low})}{\kappa_{c}}\right\}\epsilon^{-\frac{2(1+\theta)}{\theta}},
\label{eq:extra33}
\end{equation}
then (\ref{eq:4.4}) holds.
\end{theorem}

\begin{proof}
By (\ref{eq:extra32}), we have $\nu_{k}\to 0$. Moreover, $\nu_{k}\leq\sigma$ and given $\delta>0$,
\begin{equation*}
\nu_{k}\leq\delta \quad\text{if}\quad k\geq \left(\dfrac{\sigma}{\delta}\right)^{\frac{1}{\theta}}.
\end{equation*}
Denote
\begin{equation*}
C=\sigma\quad\text{and}\quad k_{0}(\delta)=\left(\dfrac{\sigma}{\delta}\right)^{\frac{1}{\theta}}.
\end{equation*}
Taking $\delta=\kappa_{c}\epsilon^{2}/2$, it follows from (\ref{eq:extra33}) that
\begin{eqnarray*}
T&\geq & \max\left\{\left(\dfrac{4}{\kappa_{c}}\right)^{\frac{1+\theta}{\theta}}\sigma^{\frac{1}{\theta}}\epsilon^{-\frac{2(1+\theta)}{\theta}},1+\left(\dfrac{4\sigma}{\kappa_{c}}\right)^{\frac{1}{\theta}}\epsilon^{-\frac{2}{\theta}},\dfrac{2(f(x_{0})-f_{low})}{\kappa_{c}}\epsilon^{-2}\right\}\\
 & = &\max\left\{\dfrac{2k_{0}(\delta/2)C}{\delta},1+k_{0}(\delta/2),\dfrac{2(f(x_{0})-f_{low})}{\kappa_{c}\epsilon^{2}}\right\}.
\end{eqnarray*}
Thus, by Corollary \ref{cor:2.2}, (\ref{eq:4.4}) must be true.
\end{proof}

\begin{remark}
\label{rem:geninho}
The smaller is $\theta$, the bigger is the chance to accept $x_{k,\ell}^{+}$ with $f(x_{k,\ell}^{+})>f(x_{k})$. Thus, the higher level of non-monotonicity obtained with $\theta\in (0,1]$ may lead to better local minimizers. However, this has a price: by Theorem \ref{thm:extra6}, the number of iterations that Algorithm 1 needs to find approximate stationary points may be significantly bigger in comparison to the case $\theta>1$.
\end{remark}

\section{Preliminary Numerical Experiments}
We performed some numerical experiments comparing Octave implementations of six instances of Algorithm 1. Specifically, we considered the following codes:
\begin{itemize}
\item[(i)] the monotone algorithm obtained from Algorithm 1 by setting $\nu_{k,l}=0$ for all $k$ and $l$. We shall refer to this code as ``M1''.
\item[(ii)] the non-monotone algorithm in \cite{GRI} obtained from Algorithm 1 by setting $\nu_{k,l}=\max_{0\leq j\leq m_{k}}\left[f(x_{k-j})\right]-f(x_{k})$ for all $k$ and $l$, where $m(0)=0$ and $m(k)=\min\left[m(k-1)+1,10\right]$. We shall refer to this code as ``NM1''.
\item[(iii)] the non-monotone algorithm in \cite{ZHA} obtained from Algorithm 1 by setting $\nu_{k,l}=C_{k}-f(x_{k})$ for all $k$ and $l$, where $C_{0}=f(x_{0})$ and, for all $k\geq 1$,
\begin{equation*}
C_{k}=\dfrac{\eta_{k-1}Q_{k-1}C_{k-1}+f(x_{k})}{Q_{k}}\quad\forall k\geq 1,
\end{equation*}
$Q_{k}=\eta_{k-1}Q_{k-1}+1$ and $\eta_{k-1}=0.85/k$, with $Q_{0}=1$. We shall refer to this code as ``NM2''.
\item[(iv)] the non-monotone algorithm obtained from Algorithm 1 by setting $\nu_{0,l}=0$, and $\nu_{k,l}=\epsilon/k$ for $k\geq 1$, where $\epsilon$ is the desired precision for the norm of the gradient. We will refer to this code as ``NM3''.
\item[(v)] the non-monotone algorithm obtained from Algorithm 1 by setting $\nu_{0,l}=0$, and $\nu_{k,l}=\gamma_{k}\|\nabla f(x_{k})\|_{2}^{2}$ with $\gamma_{k}=\|\nabla f(x_{0})\|_{2}^{-2}\left(\frac{1}{k}\right)$, for $k\geq 1$. We will refer to this code as ``NM4''.
\item[(vi)] the non-monotone algorithm obtained from Algorithm 1 by setting $\nu_{k,l}$ as in (\ref{eq:4.2}), with user-define positive parameters $\sigma$ and $\theta$. We shall refer to this code as ``NM5($\sigma,\theta$)''.
\end{itemize}
In all implementations, we consider the parameters $\alpha_{0}=1$ and $\beta=\rho=0.5$. The search directions were generated as $d_{k}=-H_{k}\nabla f(x_{k})$, where $H_{k}$ is computed using the BFGS update whenever it is possible, namely:
\begin{equation*}
H_{k+1}=\left\{\begin{array}{ll} \left(I-\dfrac{s_{k}y_{k}^{T}}{s_{k}^{T}y_{k}}\right)H_{k}\left(I-\dfrac{y_{k}s_{k}^{T}}{s_{k}^{T}y_{k}}\right)+\dfrac{s_{k}s_{k}^{T}}{s_{k}^{T}y_{k}},&\text{if}\,\, s_{k}^{T}y_{k}>0,\\
                                 H_{k},           & \text{otherwise}.
						   \end{array}
				\right.
\end{equation*}
where $H_{0}=I$, $s_{k}=x_{k+1}-x_{k}$ and $y_{k}=\nabla f(x_{k+1})-\nabla f(x_{k})$. All the experiments were performed with Octave 4.2.2 on a PC with a 2.70 GHz Intel(R) i5 microprocessor. 

In our first experiment, we applied the referred codes to the set of problems from \cite{MOR}\footnote{We considered the same dimensions as in \cite{GS}.}. We used the stopping rules:
\begin{equation}
\|\nabla f(x_{k})\|_{2}\leq 10^{-5},
\label{eq:stela1}
\end{equation}
and
\begin{equation}
k=k_{\max}\equiv 500.
\label{eq:stela2}
\end{equation}
We declare that a problem $p$ was solved by a solver $s$ when $s$ stopped due to (\ref{eq:stela1}). In this case, let $$
n_{p,s}=\text{the number of iterations required to solve problem $p$ by solver $s$.}
$$ 
As proposed in \cite{DOL}, the relative performance of solver $s$ on problem $p$ can be measured in terms of the \textit{performance ratio}
\begin{equation*}
r_{p,s}=\dfrac{n_{p,s}}{n_{p}^{*}},\quad\text{where}\,\,n_{p}^{*}=\min\left\{n_{p,s}\,:\,s\in\mathcal{S}\right\}.
\end{equation*}
Using $r_{p,s}$, the \textit{performance profile} for each code $s$ is defined as
\begin{equation*}
\rho_{s}(\tau)=\dfrac{\text{no. of problems s.t. $r_{p,s}\leq\tau$}}{\text{total no. of problems}}.
\end{equation*}
Note that $\rho_{s}(1)$ is the percentage of problems for which the solver $s$ wins over the rest of the solvers (i.e., $n_{p,s}=n_{p}^{*}$). The usual graphs of performance profiles are not very informative in this case because of the superposition of a large number of lines (one for each code), which makes difficult the interpretation of the results. Therefore, we summarize the relevant information at Table \ref{table:01}. Specifically, we report the performance ratio $r_{p,s}$ for each pair $(p,s)$ in our test set, and $\rho_{s}(1)$ for each solver $s$. 

\newpage
\noindent An entry ``F'' indicates that the corresponding code stopped due to (\ref{eq:stela2}). As we can see, solver NM4 was the most efficient (winning on $60.6\%$ of the problems), while solvers NM1 and NM3 were the most robust (failing in only two problems). The superior performance of NM2 and NM4 are in line with the empirical evidence that it is better to start with a bigger non-monotone term far from a critical point and to have a smaller one close to it (see, e.g., \cite{AHO,NOS}). Moreover, the fact that NM5($\epsilon,2$) was more efficient than NM5($\epsilon,1$) is in accordance with Theorems \ref{thm:5.1} and \ref{thm:extra6} (recall Remark \ref{rem:geninho}). These numerical results illustrate the ability of our new nonmonotone rules for finding approximate critical points. 

\begin{table}
\centering
\small{    
\begin{tabular}{ | l | c | c | c | c | c | c | c |}
    \hline
 \textbf{PROBLEM ($n_{p}^{*}$)} &\textbf{M1}& \textbf{NM1} & \textbf{NM2} & \textbf{NM3}& \textbf{NM4} & \textbf{NM5($\epsilon,2$)} & \textbf{NM5($\epsilon,1$)} \\ \hline
    1.  ($43$)             & 1.000 & 1.162 & 1.000 & 1.558 & 1.000 & 1.581 & 1.534 \\ 
    2.  ($27$)             & 1.037 & 4.629 & 1.037 & 1.666 & 1.000 & 1.444 & 1.629 \\ 
    3.  ($173$)            & 1.017 & 1.034 & 1.017 & F     & 1.000 & F     & F     \\ 
    4.  ($50$)             & 1.480 & 4.680 & 1.140 & 1.260 & 1.000 & 1.300 & 1.260 \\ 
    5.  ($17$)             & 1.058 & 1.529 & 1.058 & 1.764 & 1.000 & 2.235 & 1.705 \\ 
    6.  ($36$)             & 1.000 & 5.666 & 1.000 & 3.083 & 1.000 & 1.027 & 3.055 \\ 
    7.  ($34$)             & 1.000 & 1.323 & 1.000 & 1.205 & 1.000 & 1.676 & 1.176 \\ 
    8.  ($19$)             & 1.000 & 2.473 & 1.000 & 2.421 & 1.000 & 1.105 & 2.368 \\ 
    9.  ($3$)              & 1.333 & 1.333 & 1.333 & 1.333 & 1.333 & 1.000 & 1.000 \\ 
    10. ($305$)            & F     & F     & F     & F     & F     & 1.000 & F     \\ 
    11. ($9$)              & 1.000 & 1.111 & 1.000 & 10.000& 1.000 & 1.111 & 10.000\\ 
    12. ($25$)             & 1.000 & 3.240 & 1.200 & 4.440 & 1.000 & 1.120 & 4.120 \\ 
    13. ($39$)             & 1.000 & 1.102 & 1.000 & 6.000 & 1.000 & 1.025 & 6.461 \\ 
    14. ($85$)             & 1.000 & 1.729 & 1.000 & 1.741 & 1.000 & 1.023 & 1.729 \\ 
    15. ($24$)             & 1.000 & 1.166 & 1.000 & 1.666 & 5.416 & 1.083 & 1.625 \\
    16. ($83$)             & F     & 1.108 & F     & 1.698 & F     & 1.000 & 1.686 \\
    17. ($79$)             & 1.000 & 1.075 & 1.000 & 1.050 & 1.000 & F     & F     \\
    18. ($37$)             & 1.054 & 1.081 & 1.054 & 1.000 & 1.054 & 1.432 & 5.648 \\
    19. ($49$)             & 1.081 & 1.000 & 1.081 & 1.612 & 1.040 & 1.428 & 1.591 \\
    20. ($57$)             & 1.000 & 1.263 & 1.000 & 1.666 & 1.000 & 2.368 & 1.649 \\
    21. ($67$)             & 1.000 & 1.910 & 1.000 & 1.835 & 1.000 & 1.194 & 1.820 \\
    22. ($39$)             & 1.000 & 1.102 & 1.000 & 6.000 & 1.000 & 1.025 & 6.461 \\
    23. ($63$)             & 1.158 & 4.031 & 1.158 & 2.857 & 1.158 & 1.000 & 2.857 \\
    24. ($16$)             & 1.187 & 1.187 & 1.187 & 5.562 & 1.000 & 1.375 & 5.500 \\
    25. ($60$)             & 1.000 & 3.350 & 1.000 & 1.650 & 1.000 & 1.083 & 1.633 \\
    26. ($58$)             & F     & F     & F     & 1.000 & F     & F     & 1.000 \\
    27. ($10$)             & 1.900 & 1.000 & 1.900 & 3.300 & F     & 1.100 & 1.400 \\
    28. ($8$)              & 1.750 & 1.500 & 1.000 & 2.500 & 1.125 & 4.125 & 2.750 \\
    29. ($31$)             & 1.032 & 1.709 & 1.000 & 1.032 & 1.032 & 2.193 & 1.032 \\
    30. ($38$)             & 1.078 & 1.473 & 1.000 & 5.078 & 1.078 & 2.342 & 5.052 \\
    31. ($1$)              & 4.000 & 3.000 & 3.000 & 3.000 & 3.000 & 1.000 & 1.000 \\
    32. ($20$)             & 1.000 & 1.000 & 1.000 & 1.000 & 1.000 & 1.000 & 1.000 \\
    33. ($18$)             & 1.055 & 1.000 & 1.055 & 1.000 & 1.000 & 1.000 & 1.000 \\\hline
    $\rho_{s}(1)$   & 0.454 & 0.121 & 0.515 & 0.121 & 0.606 & 0.212 & 0.151\\\hline
    \end{tabular}
}
\caption{Numerical results for problems from the Mor\'e-Garbow-Hillstrom collection \cite{MOR}.}
\label{table:01}
\end{table}

In order to investigate the ability of non-monotone methods for finding better local optima, we applied all codes to minimize the two-dimensional Griewank function \cite{GRIE}:
\begin{equation}
f(x)=1+\dfrac{x_{1}^{2}}{4000}+\dfrac{x_{2}^{2}}{4000}-\cos(x_{1})\cos\left(x_{2}/\sqrt{2}\right).
\label{eq:5.2}
\end{equation}
This function has a huge number of local minimizers but only one global minimizer, namely $x^{*}=(0,0)$ with $f(x^{*})=0$. We considered 60 initial points generated in the box $\left[-600,600\right]\times\left[-600,600\right]$:
\begin{equation*}
\left(-600+\frac{1200(i-1)}{3},-600+\frac{1200(j-1)}{14}\right),\quad i=1,\ldots,4,\,\,j=1,\ldots,15.
\end{equation*}
For each starting point we recorded the best function value found by each code within 500 iterations. The distributions of these values are summarized in Table \ref{table:02}.

\begin{table}[ht]
\centering    
\begin{tabular}{ | l | c | c | c | c | c | c | }
\hline
                    & \textbf{M1}          & \textbf{NM1}        & \textbf{NM2}        & \textbf{NM3}        & \textbf{NM4}        & \textbf{NM5($\epsilon,2$)}        \\ \hline
    Maximum         & $179.8002$ & $136.3502$ & $179.8002$ & $179.8002$ & $179.8002$ & $179.8002$ \\\hline 
    75th Percentile & $119.1955$ & $89.9534$  & $119.1955$ & $119.1955$ & $119.1955$ & $119.1955$ \\\hline
    Median          & $82.7324$  & $25.2736$  & $82.7324$  & $82.7324$  & $78.1701$  & $82.7324$  \\\hline 
    25th Percentile & $34.0983$  & $9.7496$   & $28.9691$  & $34.0983$  & $34.0983$  & $34.0983$  \\\hline
    Minimum         & $10.1014$  & $0.3353$   & $10.1014$  & $10.1014$  & $10.1014$  & $10.1014$  \\\hline 
    \end{tabular}
\caption{Results for the Griewank function: distributions of the best function values found by each code within 500 iterations.}
\label{table:02}
\end{table}

From Table \ref{table:02}, we see that code NM1 (with more aggressive non-monotone behavior) was much better than the other codes in terms of the best function values found. Since in NM5 the non-monotonicity can be increased by increasing $\sigma$ or decreasing $\theta$, we set $\sigma=|f(x_{0})|$ and tested several values of $\theta$. The distributions of the best function values found by each variant of NM5 within 500 iterations are summarized on Table \ref{table:03}.

\begin{table}[ht]
\centering    
\begin{tabular}{ | l | c | c | c | c | c | c | }
\hline
 \textbf{NM5($|f(x_{0})|,\theta$)} & $\theta=4$ & $\theta=2$ & $\theta=1$ & $\theta=0.5$ & $\theta=0.25$ & $\theta=0.125$\\\hline
    Maximum         & $136.3843$ & $124.3656$ & $70.2559$ & $19.2514$ & $10.1188$ & $18.2874$ \\\hline 
    75th Percentile & $99.6332$  & $96.3803$  & $29.7006$ & $6.8724$  & $4.9171$  & $2.6889$ \\\hline
    Median          & $62.0849$  & $70.6839$  & $19.2193$ & $2.1444$  & $1.6082$  & $0.9238$  \\\hline 
    25th Percentile & $34.0983$  & $19.2036$  & $10.1014$ & $0.7201$  & $0.3102$  & $0.2367$  \\\hline
    Minimum         & $10.1014$  & $5.8595$   & $1.1145 $ & $0.0377$  & $0.0404$  & $0.0609$  \\\hline 
    \end{tabular}
\caption{Results for code NM5 with $M=|f(x_{0})|$ and different values of $\theta$.}
\label{table:03}
\end{table}
As expected, the best function values were obtained with small values of $\theta$ (see Remark \ref{rem:geninho}). Moreover, the function values obtained with $\theta\leq 0.5$ were significantly better than the values obtained with NM1. These preliminary results confirm the ability of non-monotone methods of escaping from the closest local minimizers. Moreover, they suggest that non-monotone line searches based on the Metropolis rule (as in NM5) may be competitive with standard non-monotone methods on difficult problems with many non-global local minimizers.

\section{Conclusion}

In this paper, we investigated the worst-case complexity of a generalized version of the non-monotone line search framework proposed in \cite{SAC} for smooth unconstrained optimization problems. In this framework, the level of non-monotonicity is controlled by a sequence $\left\{\nu_{k}\right\}$ of non-negative parameters. In a previous paper \cite{GS}, we proved that the algorithms in the referred framework take at most $\mathcal{O}(\epsilon^{-2})$ iterations to find $\epsilon$-critical points, when the objective $f$ is nonconvex. For that, we had to assume that $\sum_{k=0}^{+\infty}\nu_{k}<+\infty$. Now, by refining our analysis, we were able to obtain bounds of the same order even when $\sum_{k=0}^{+\infty}\nu_{k}=+\infty$. Our generalized results include a unified global convergence proof for non-monotone schemes in which $\nu_{k}\to 0$, allowing more freedom for the design of new non-monotone line search algorithms. As a topic for future research, it would be interesting to investigate the possible extension of our results to inexact subsampled methods for minimizing finite sums \cite{BEL1,BEL2}. 

\section*{Acknowledgements}
We are very grateful to the three anonymous referees, whose comments helped to improve significantly the paper. We are also grateful to Masoud Ahookhosh for his insightful comments on the first version of this paper.



\end{document}